\documentclass[10pt]{article}

\usepackage{latexsym,amsmath,amsfonts,amsthm}

\usepackage{a4wide}

\newcommand{\R}{\mathbb{R}}
\newcommand{\T}{\mathbb{T}}

\newtheorem{lemma}{Lemma}[section]
\newtheorem{theorem}[lemma]{Theorem}
\newtheorem{prop}[lemma]{Proposition}

\newtheorem{df}[lemma]{Definition}

\newtheorem{remark}[lemma]{Remark}

\begin{document}

\title{Noether's symmetry theorem for nabla problems\\
of the calculus of variations\thanks{Submitted 20/Oct/2009; 
Revised 27/Jan/2010; Accepted 28/July/2010; 
for publication in \emph{Appl. Math. Lett.}}}

\author{Nat\'{a}lia Martins\\
\texttt{natalia@ua.pt}
\and
Delfim F. M. Torres\\
\texttt{delfim@ua.pt}}

\date{Department of Mathematics\\
      University of Aveiro\\
      3810-193 Aveiro, Portugal}

\maketitle


\begin{abstract}
We prove a Noether-type symmetry theorem and
a DuBois-Reymond necessary optimality condition
for nabla problems of the calculus of variations
on time scales.
\end{abstract}

\smallskip


\smallskip


\smallskip



\section{Introduction}

The theory of time scales was born with the 1989 PhD thesis
of Stefan Hilger, done under supervision of Bernd Aulbach
\cite{Aulbach:Hilger}. The aim was to unify various
concepts from the theories of discrete and continuous dynamical
systems, and to extend such theories to more general classes of dynamical systems.
The calculus of time scales is nowadays a powerful tool,
with two excellent books dedicated to it \cite{Bh,Bh:03}.
For a good introductory survey on time scales
we refer the reader to \cite{survey:ts}.

The calculus of variations is well-studied in the continuous, discrete, and quantum settings
(see, \textrm{e.g.}, \cite{Bang:q-calc,GelfandFomin,KP}).
Recently an important and very active line of research has been unifying and generalizing the known calculus
of variations on $\mathbb{R}$, $\mathbb{Z}$, and $q^{\mathbb{N}_0}:=\{q^k | k \in \mathbb{N}_0\}$, $q>1$,
to an arbitrary time scale $\mathbb{T}$ via delta calculus. Progress toward this has been made on the topics of
necessary and sufficient optimality conditions and its applications -- see
\cite{NT:ts,bohner:CV,Rui:Del,Mal:Tor:09,Mal:Tor:Wei} and references therein.
The goal is not to simply reprove existing and well-known theories,
but rather to view $\mathbb{R}$, $\mathbb{Z}$, and $q^{\mathbb{N}_0}$
as special cases of a single and more general theory.
Doing so reveals richer mathematical structures (\textrm{cf.} \cite{bohner:CV})
which has great potential for new applications, in particular
in engineering \cite{SSW} and economics \cite{Atici:McMahan,Atici:Uysal:08}.

The theory of time scales is, however, not unique.
Essentially, two approaches are followed in the literature:
one dealing with the delta calculus (the forward approach) \cite{Bh};
the other dealing with the nabla calculus (the backward approach) \cite[Chap.~3]{Bh:03}.
To actually solve problems of the calculus of variations and optimal control
it is often more convenient to work backwards in time,
and recently a general theory of the calculus of variations
on time scales was introduced via the nabla
operator. Results include: Euler-Lagrange necessary optimality
conditions \cite{Atici}, necessary conditions for higher-order nabla problems
\cite{NataliaHigherOrderNabla}, and optimality conditions for variational problems subject to
isoperimetric constraints \cite{comRicardoISO:nabla}.
In this note we develop further the theory by
proving two of the most beautiful results
of the calculus of variations --- the Noether
symmetry theorem  and the DuBois-Reymond condition
\cite{Torres04} --- to nabla variational
problems on an arbitrary time scale $\T$.
Our main tool is the recent duality technique
of M.~C.~Caputo \cite{Caputo}, which allows
to obtain nabla results on time scales
from the delta theory. Caputo's duality concept is briefly
presented in Sec.~\ref{sec:prel}; in Sec.~\ref{sec:mr}
our results are formulated and proved; in Sec.~\ref{sec:ex}
an illustrative example is given. We end with
some words about the originality of our results
and the state of the art (Sec.~\ref{sec:fc}).


\section{Preliminaries}
\label{sec:prel}

We assume the reader to be familiar with the calculus on time scales \cite{Bh,Bh:03}.
Here we just review the main tool used in the paper: duality.

Let $\T$ be an arbitrary time scale and let $\T^\ast:=\{s\in\R:-s\in\T\}$.
The new time scale $\T^\ast$ is called the \emph{dual time scale} of $\T$.
If $\sigma$ and $\rho$ denote, respectively, the forward and backward
jump operators on $\T$, then we denote by $\widehat{\sigma}$ and $\widehat{\rho}$
the forward and backward jump operators of $\T^\ast$. Similarly,
if $\mu$ and $\nu$ denote, respectively, the forward and backward graininess function on $\T$,
then $\widehat{\mu}$ and $\widehat{\nu}$ denote, respectively,
the forward and backward graininess function on $\T^\ast$; if $\Delta$ (resp. $\nabla$)
denote the delta (resp. nabla) derivative on $\T$, then $\widehat{\Delta}$ (resp. $\widehat{\nabla}$)
will denote the delta (resp. nabla) derivative on $\T^\ast$.

\begin{df}
Given a function $f:\T\rightarrow \R$ defined on time scale $\T$
we define the dual function $f^{\ast}:\T^{\ast}\rightarrow\R$
by $f^{\ast}(s):=f(-s)$ for all $s\in\T^{\ast}$.
\end{df}

We recall some basic results concerning
the relationship between \emph{dual objects}.
The set of all rd-continuous (resp. ld-continuous)
functions is denoted by $C_{rd}$ (resp. $C_{ld}$).
Similarly, $C^1_{rd}$ (resp. $C^1_{ld}$) will denote the set
of functions from $C_{rd}$ (resp. $C_{ld}$)
whose delta (resp. nabla) derivative belongs to $C_{rd}$ (resp. $C_{ld}$).

\begin{prop}[\cite{Caputo}]
\label{integral}
Let $\T$ be a given a time scale with $a,b\in\T$, $a<b$,
and $f:\T\rightarrow \R$. Then,
\begin{enumerate}
\item $(\T^{\kappa})^{\ast}=(\T^{\ast})_{\kappa}$ and $(\T_{\kappa})^{\ast}=(\T^{\ast})^{\kappa}$;

\item $([a,b])^{\ast} =[-b,-a]$ and $([a,b]^{\kappa})^{\ast} = [-b,-a]_{{k}} \subseteq \T^{\ast}$;

\item for all $s\in\T^{\ast}$,
$\widehat{\sigma}(s)=-\rho(-s)=-\rho^{\ast}(s)$ and
$\widehat{\rho}(s)=-\sigma(-s)=-\sigma^{\ast}(s)$;

\item for all $s\in\T^{\ast}$, $\widehat\nu (s)=\mu^{\ast}(s)$ and $\widehat\mu (s)=\nu^{\ast}(s)$;

\item $f$ is rd (resp. ld) continuous
if and only if its dual $f^{\ast}:\T^{\ast}\rightarrow \R$
is ld (resp. rd) continuous;

\item if $f$ is \emph{delta} (resp. \emph{nabla})
differentiable at $t_0\in\T^{\kappa}$ (resp. at $t_0\in\T_{\kappa}$),
then $f^{\ast}:\T^{\ast}\rightarrow \R$ is \emph{nabla} (resp. \emph{delta})
differentiable at $-t_0\in(\T^{\ast})_{\kappa}$ (resp. $-t_0\in(\T^{\ast})^{\kappa}$), and
\begin{equation*}
\begin{split}
f^{\Delta}(t_0) &= -(f^{\ast})^{\widehat\nabla}(-t_0) \quad \text{(resp. $f^{\nabla}(t_0)=-(f^{\ast})^{\widehat\Delta}(-t_0)$)} \, ,\\
f^{\Delta}(t_0) &= -((f^{\ast})^{\widehat\nabla})^{\ast}(t_0) \quad \text{(resp. $f^{\nabla}(t_0)=-((f^{\ast})^{\widehat\Delta})^{\ast}(t_0)$)}\, ,\\
(f^{\Delta})^{\ast}(-t_0) &= -((f^{\ast})^{\widehat\nabla})(-t_0) \quad \text{(resp. $(f^{\nabla})^{\ast}(-t_0)=-(f^{\ast})^{\widehat\Delta}(-t_0)$)}\, ;
\end{split}
\end{equation*}

\item $f$ belongs to  $C^1_{rd}$ (resp. $C^1_{ld}$)
if and only if its dual $f^{\ast}:\T^{\ast}\rightarrow \R$ belongs to $C^1_{ld}$
(resp. $C^1_{rd}$);

\item if $f:[a,b]\rightarrow \R$ is rd continuous, then
$$\int_a^b f(t)\Delta t=\int_{-b}^{-a}f^{\ast}(s)\widehat\nabla s\, ;$$

\item if $f:[a,b]\rightarrow \R$ is  ld continuous, then
$$\int_a^b f(t)\nabla t=\int_{-b}^{-a}f^{\ast}(s)\widehat\Delta s\, .$$
\end{enumerate}
\end{prop}

\begin{df}
\label{def:dualL}
Given a Lagrangian  $L:\T\times \R^n\times\R^n\rightarrow \R$,
we define the corresponding dual Lagrangian
$L^{\ast}:\T^{\ast}\times \R^n \times \R^n \rightarrow \R$ by
$L^{\ast}(s,x,v)=L(-s,x,-v)$ for all $(s,x,v)\in\T^{\ast}\times\R^n\times\R^n$.
\end{df}

As a consequence of Definition~\ref{def:dualL}
and Proposition~\ref{integral} we have the following useful lemma:

\begin{lemma}
\label{corlag}
Given a continuous Lagrangian $L:\R \times \R^n\times\R^n\rightarrow \R$ one has
$$\int_a^b L\left(t, y^{\sigma}(t), y^{\Delta}(t)\right)\Delta t
= \int_{-b}^{-a} L^{\ast}\left(s, (y^{\ast})^{\widehat \rho}(s), (y^{\ast})^{\widehat\nabla}(s)\right)\widehat\nabla s$$
for all functions $y\in C^1_{rd}\left([a,b], \mathbb{R}^n\right)$.
\end{lemma}

\begin{df}
Let $\mathbb{T}$ be a given time
scale with at least three points, $n \in \mathbb{N}$,
and $L: \mathbb{R}\times \mathbb{R}^n \times \mathbb{R}^n
\rightarrow \mathbb{R}$ be of class $C^1$.
Suppose that $a,b\in \mathbb{T}$ and $a<b$.
We say that $q_0\in C_{rd}^{1}$
is a local minimizer for problem
\begin{equation}
\label{problem}
\begin{gathered}
\mathcal{I}[q]=\int_a^b L(t,q^\sigma(t),q^\Delta(t)) \Delta t
\longrightarrow \min\\
q(a)=q_a\, , \quad q(b)=q_b \, ,
\end{gathered}
\end{equation}
if there exists $\delta > 0$ such that
$$
\mathcal{I}[q_0]\leq \mathcal{I}[q]
$$
for all $q \in C_{rd}^{1}([a,b], \mathbb{R}^n)$
satisfying the boundary conditions
$q(a)=q_a$, $q(b)=q_b$ and
$$
\parallel q - q_0\parallel :=
\sup_{t \in [a,b]^\kappa}\mid q^{\sigma}(t)-q_0^{\sigma}(t)\mid
+ \sup_{t \in [a,b]^\kappa}\mid q^{\Delta}(t)-q_0^{\Delta}(t)\mid < \delta \, ,
$$
where $|\cdot|$ denotes a norm in $\mathbb{R}^n$.
\end{df}

The following result, known as DuBois-Reymond equation
or second Euler-Lagrange equation, is a necessary optimality
condition for optimal trajectories of delta variational problems.

\begin{theorem}[DuBois-Reymond equation for delta problems \cite{NatyZbig}]
\label{secondEL}
If $q\in \mathrm{C}^1_{rd}$ is a local minimizer of problem \eqref{problem},
then $q$ satisfies the equation
\begin{equation}
\label{2equationEL}
\frac{\Delta}{\Delta t} \mathcal{H}(t,q^\sigma(t),q^\Delta(t))
=-\partial_1 L(t,q^\sigma(t),q^\Delta(t))
\end{equation}
for all $t\in [a,b]^\kappa$, where
$\mathcal{H}(t,u,v)=-L(t,u,v) + \partial_3 L(t,u,v)v
+ \partial_1 L(t,u,v) \mu(t)$.
\end{theorem}


\section{Main Results}
\label{sec:mr}

Our focus is Emmy Noether's theorem,
a fundamental tool of modern theoretical physics
and the calculus of variations,
which allows to derive conserved quantities
from the existence of variational symmetries
(see, \textrm{e.g.}, \cite{Gasta,Torres02,Torres04}).
We prove here a Noether's theorem for variational problems
with nabla derivatives and integrals (Theorem~\ref{Noether nabla}).

Let $\mathbb{T}$ be a given time
scale with at least three points, $n \in \mathbb{N}$,
and $L: \mathbb{R}\times \mathbb{R}^n \times \mathbb{R}^n
\rightarrow \mathbb{R}$ be of class $C^1$.
Suppose that $a,b\in \mathbb{T}$ and $a<b$.
We consider the following nabla variational problem on $\mathbb{T}$:
\begin{equation}
\label{problem nabla}
\mathcal{I}[q]=\int_a^b L(t,q^\rho(t),q^\nabla(t)) \nabla t
\longrightarrow \min_{q\in \mathcal{Q}},
\end{equation}
where
\[
\mathcal{Q}=\{ q\ | \ q: [a,b]
\rightarrow \mathbb{R}^n,\ q\in \mathrm{C}^1_{ld},\
q(a)=A,\ q(b)=B\}
\]
for some $A, B \in \mathbb{R}^n$, and
where $\rho$ is the backward jump operator and $q^\nabla$
is the nabla-derivative of $q$ with respect to $\mathbb{T}$.
Let $V=\{q \  | \  q:[a,b] \rightarrow \mathbb{R}^n$,
$q \in C^1_{ld}\}$, and consider a one-parameter
family of infinitesimal transformations
\begin{equation}
\label{eq:tinf nabla}
\begin{cases}
\bar{t} = T(t,q, \epsilon) = t + \epsilon\tau(t,q) + o(\epsilon) \, ,\\
\bar{q} = Q(t,q, \epsilon) = q + \epsilon\xi(t,q) + o(\epsilon) \, ,\\
\end{cases}
\end{equation}
where $\epsilon$ is a small real parameter,
and $\tau:[a,b] \times \mathbb{R}^n\rightarrow\mathbb{R}$
and $\xi:[a,b] \times \mathbb{R}^n\rightarrow\mathbb{R}^n$
are nabla differentiable functions.
We assume that for every $q$ and every $\epsilon$ the map
$[a,b]\ni t \mapsto \alpha(t):= T(t,q(t), \epsilon)\in\mathbb{R}$ is a
strictly increasing $\mathrm{C}^1_{ld}$ function and its image is
again a time scale with backward shift operator $\overline{\rho}$ and nabla
derivative $\overline{\nabla}$.

\begin{df}
\label{def. invariance nabla}
Functional $\mathcal{I}$ in \eqref{problem nabla} is said
to be invariant on $V$ under the family of
transformations \eqref{eq:tinf nabla} if
$$
\frac{d}{d\epsilon}\left\{L\left(T(t,q(t), \epsilon),
Q^{\rho}(t,q(t), \epsilon), \frac{Q^{\nabla}(t,q(t),\epsilon)}{T^{\nabla}(t,q(t), \epsilon)}\right)
T^{\nabla}(t,q(t), \epsilon)\right\}\Big|_{\epsilon=0} = 0 \, .
$$
\end{df}

\begin{remark}
\label{thm:invariance nabla}
Functional $\mathcal{I}$ in \eqref{problem nabla}
is invariant on $V$ under the family
of transformations \eqref{eq:tinf nabla} if and only if
\begin{multline*}
\partial_{1}L(t,q^{\rho}(t), q^{\nabla}(t))\tau(t,q(t))
+ \partial_{2}L(t,q^{\rho}(t), q^{\nabla}(t))\xi^{\rho}(t,q(t))\\
+\partial_{3}L(t,q^{\rho}(t), q^{\nabla}(t))\xi^{\nabla}(t,q(t))
+L(t,q^{\rho}(t), q^{\nabla}(t))\tau^{\nabla}(t,q(t))\\
-q^{\nabla}(t)\partial_{3}L(t,q^{\rho}(t), q^{\nabla}(t))\tau^{\nabla}(t,q(t))=0
\end{multline*}
for all $t \in [a,b]_\kappa$ and all $q \in V$,
where $\partial_{i} L$ denotes the partial
derivative of $L(\cdot,\cdot,\cdot)$
with respect to its $i$-th argument, $i = 1, 2, 3$, and
$$
\xi^{\rho}(t,q(t))=\xi(\rho(t), q(\rho(t))) \, ,
\quad
\xi^{\nabla}(t,q(t))=\frac{\nabla}{\nabla t}\xi(t,q(t))\, .
$$
\end{remark}

\begin{df}
We say that function $q\in \mathrm{C}^1_{ld}$ is an extremal
of problem \eqref{problem nabla} if it satisfies the
nabla Euler-Lagrange equation
\begin{equation}
\label{eq:nabla:EL:eq}
\partial_3 L(t,q^{\rho}(t), q^{\nabla}(t))
-\int_{a}^{t} \partial_2 L(\tau,q^{\rho}(\tau), q^{\nabla}(\tau)) \nabla\tau
= \text{const} \quad \forall t \in [a,b]_\kappa \, .
\end{equation}
\end{df}

\begin{theorem}[Noether's theorem for nabla variational problems]
\label{Noether nabla}
If functional $\mathcal{I}$ in \eqref{problem nabla} is invariant on
$V$ in the sense of Definition~\ref{def. invariance nabla}, then
\begin{equation*}
\partial_{3}L(t,q^{\rho},q^{\nabla})\cdot\xi(t,q)
+ \Bigl[L(t,q^{\rho},q^{\nabla})
-\partial_{3}L(t,q^{\rho},q^{\nabla})\cdot q^{\nabla}
+ \partial_{1}L(t,q^{\rho},q^{\nabla})\cdot\nu(t)\Bigr]\cdot\tau(t,q)
\end{equation*}
is constant along all the extremals of problem {\rm (\ref{problem nabla})}.
\end{theorem}

\begin{proof} Let $q_0$ be an extremal of problem \eqref{problem nabla}. Then $q_0^\ast$ is an extremal of problem
\begin{gather*}
\mathcal{I}^\ast[g]=\int_{-b}^{-a} L^\ast(t,g^{\widehat{\sigma}}(t),g^{\widehat{\Delta}}(t)) \widehat{\Delta} t
\longrightarrow \min_{g\in C^1_{rd}}\\
g(-b)=B \, , \quad g(-a)=A \, ,
\end{gather*}
\textrm{i.e.},
\begin{equation*}
\partial_3 L^\ast\left(t,(q_0^\ast)^{\widehat{\sigma}}\left(t\right),(q_0^\ast)^{\widehat{\Delta}}(t)\right)
-\int_{-b}^{t} \partial_2 L^\ast\left(\tau,(q_0^\ast)^{\widehat{\sigma}}\left(\tau\right),(q_0^\ast)^{\widehat{\Delta}}(\tau)\right)
\widehat{\Delta} \tau = \text{const} \quad \forall t \in [-b,-a]^\kappa \, .
\end{equation*}
Now we note that if $\mathcal{I}$ is invariant on $V$ under the family of
transformations \eqref{eq:tinf nabla}, then $\mathcal{I}^\ast$ is invariant on
$\overline{U}=\{g \  | \  g:[-b,-a] \rightarrow \mathbb{R}^n,
g \in C^1_{rd}\}$ under the family of transformations
$$
\begin{cases}
\bar{t} =  t - \epsilon\tau^\ast(t,g) + o(\epsilon) \, ,\\
\bar{g} =  g + \epsilon\xi^\ast(t,g) + o(\epsilon) \, ,\\
\end{cases}
$$
where $\tau^\ast(t,u)=\tau(-t,u)$ and $\xi^\ast(t,u)=\xi(-t,u)$.
Hence, by \cite[Theorem~4]{NT:ts} on delta problems we can conclude that
\begin{multline*}
\partial_{3}L^\ast(t,(q_0^\ast)^{\widehat{\sigma}}(t),(q_0^\ast)^{\widehat{\Delta}}(t))\cdot\xi^\ast(t,q_0^\ast(t))
+ \Bigl[L^\ast(t,(q_0^\ast)^{\widehat{\sigma}}(t),(q_0^\ast)^{\widehat{\Delta}}(t))\\
-\partial_{3}L^\ast(t,(q_0^\ast)^{\widehat{\sigma}}(t),(q_0^\ast)^{\widehat{\Delta}}(t))\cdot (q_0^\ast)^{\widehat{\Delta}}(t)
-\partial_{1}L^\ast(t,(q_0^\ast)^{\widehat{\sigma}}(t),(q_0^\ast)^{\widehat{\Delta}}(t))\cdot\widehat{\mu}(t)\Bigr]\cdot
\left(-\tau^\ast(t,q_0^\ast(t))\right)
\end{multline*}
is a constant. Having in mind the equalities
\begin{gather*}
(q_0^\ast)^{\widehat{\Delta}}(t) = -q_0^{\nabla}(-t)\, , \quad
\left(q_0^\ast\right)^{\widehat{\sigma}}(t) = q_{0}^{\rho}(-t) \, , \quad
\widehat{\mu}(t)=\nu(-t)\, , \\
\partial_{1}L^\ast(t,(q_0^\ast)^{\widehat{\sigma}}(t),(q_0^\ast)^{\widehat{\Delta}}(t)) = - \partial_{1}L(-t,q_0^{\rho}(-t),q_0^{\nabla}(-t))\, ,\\
\partial_{3}L^\ast(t,(q_0^\ast)^{\widehat{\sigma}}(t),(q_0^\ast)^{\widehat{\Delta}}(t)) = - \partial_{3}L(-t,q_0^{\rho}(-t),q_0^{\nabla}(-t))\, ,\\
L^\ast(t,(q_0^\ast)^{\widehat{\sigma}}(t),(q_0^\ast)^{\widehat{\Delta}}(t))= L(-t,q_0^{\rho}(-t),q_0^{\nabla}(-t))\, ,\\
\tau^\ast(t, q_0^\ast(t))=\tau(-t,q_0(-t))\, ,\\
\xi^\ast(t, q_0^\ast(t))=\xi(-t,q_0(-t))\, ,
\end{gather*}
we obtain that
\begin{multline*}
-\partial_{3}L(-t,q_0^{\rho}(-t),q_0^{\nabla}(-t))\cdot\xi(-t,q_0(-t))
+ \Bigl[L(-t,q_0^{\rho}(-t),q_0^{\nabla}(-t))
-\partial_{3}L(-t,q_0^{\rho}(-t),q_0^{\nabla}(-t))\cdot q_0^{\nabla}(-t)\\
+ \partial_{1}L(-t,q_0^{\rho}(-t),q_0^{\nabla}(-t))\cdot\nu(-t)\Bigr]\cdot\left(-\tau(-t,q_0(-t))\right)
\end{multline*}
is constant. Let $s \in [a,b]_\kappa$ and set $s=-t$. Then,
\begin{multline*}
-\partial_{3}L(s,q_0^{\rho}(s),q_0^{\nabla}(s))\cdot\xi(s,q_0(s))
+ \Bigl[L(s,q_0^{\rho}(s),q_0^{\nabla}(s))
-\partial_{3}L(s,q_0^{\rho}(s),q_0^{\nabla}(s))\cdot q_0^{\nabla}(s)\\
+ \partial_{1}L(s,q_0^{\rho}(s),q_0^{\nabla}(s))\cdot\nu(s)\Bigr]\cdot\left(-\tau(s,q_0(s))\right)
\end{multline*}
is constant, which proves the desired result.
\end{proof}

Noether's theorem explains all conservation
laws of mechanics. However, the most important conservation
law --- \emph{conservation of energy},
which is obtained in mechanics from Noether's theorem
and invariance with respect to time translations  --
is typically obtained in the calculus of variations
as a corollary of the DuBois-Reymond condition \cite{Torres04}.
We now obtain a nabla version of DuBois-Reymond condition on time scales.

\begin{df}
We say that $q_0\in \mathcal{Q}$
is a local minimizer for problem \eqref{problem nabla}
if there exists $\delta > 0$ such that
$$
\mathcal{I}[q_0]\leq \mathcal{I}[q]
$$
for all $q \in \mathcal{Q}$
satisfying
$$
\parallel q - q_0\parallel :=
\sup_{t \in [a,b]_\kappa}\mid q^{\rho}(t)-q_0^{\rho}(t)\mid
+ \sup_{t \in [a,b]_\kappa}\mid q^{\nabla}(t)-q_0^{\nabla}(t)\mid < \delta \, ,
$$
where $|\cdot|$ denotes a norm in $\mathbb{R}^n$.
\end{df}

\begin{theorem}[DuBois-Reymond condition for nabla variational problems]
\label{secondEL nabla}
If $q\in \mathcal{Q}$ is a local minimizer of problem \eqref{problem nabla},
then $q$ satisfies the equation
\begin{equation*}
\frac{\nabla}{\nabla t} \overline{\mathcal{H}}(t,q^\rho(t),q^\nabla(t))
=-\partial_1 L(t,q^\rho(t),q^\nabla(t))
\end{equation*}
for all $t\in [a,b]_\kappa$, where
$$
\overline{\mathcal{H}}(t,u,v)=-L(t,u,v) + \partial_3 L(t,u,v) \cdot v
- \partial_1 L(t,u,v) \nu(t) \, ,
$$
$t\in\T$, and $u,v\in \mathbb{R}^n$.
\end{theorem}

\begin{proof}
Let $q_0$ be local minimizer of problem \eqref{problem nabla}.
Then $q_0^\ast$ is a local minimizer of problem
\begin{equation*}
\mathcal{I}^\ast[g]=\int_{-b}^{-a} L^\ast(t,g^{\widehat{\sigma}}(t),g^{\widehat{\Delta}}(t)) \widehat{\Delta} t
\longrightarrow \min_{g\in C^1_{rd}}
\end{equation*}
subject to $g(-b)=B$ and $g(-a)=A$, $t \in [-b,-a]^\kappa$.
By the second Euler-Lagrange equation for delta problems \eqref{2equationEL} we conclude that
\begin{equation}
\label{eq-hamiltonian-dual}
\frac{\widehat{\Delta}}{\widehat{\Delta} t} \mathcal{H}(t,(q_0^\ast)^{\widehat{\sigma}}(t),(q_0^\ast)^{\widehat{\Delta}}(t))
=-\partial_1 L^\ast(t,(q_0^\ast)^{\widehat{\sigma}}(t),(q_0^\ast)^{\widehat{\Delta}}(t))
\end{equation}
for all $t\in [-b,-a]^\kappa$, where
$$
\mathcal{H}(t,u,v)=-L^\ast(t,u,v) + \partial_3 L^\ast(t,u,v) \cdot v
+ \partial_1 L^\ast(t,u,v) \widehat{\mu}(t) \, .$$
Note that
$$\mathcal{H}(t,(q_0^\ast)^{\widehat{\sigma}}(t),(q_0^{\ast})^{\widehat{\Delta}}(t))
= {\overline{\mathcal{H}}}^{\ast} (t,(q_0^\ast)^{\widehat{\sigma}}(t),(q_0^\ast)^{\widehat{\Delta}}(t))
$$
with $\mathcal{H}^\ast(t,u,v)=\mathcal{H}(-t,u,-v)$.
Since
$$
(\overline{\mathcal{H}}^{\ast})^{\widehat{\Delta}} (t,(q_0^\ast)^{\widehat{\sigma}}(t),(q_0^\ast)^{\widehat{\Delta}}(t))
= - \overline{\mathcal{H}}^{\nabla}(-t,q_0^{\rho}(-t),q_0^{\nabla}(-t))
$$
and
$$
\partial_{1}L^\ast(t,(q_0^\ast)^{\widehat{\sigma}}(t),(q_0^\ast)^{\widehat{\Delta}}(t))
= - \partial_{1}L(-t,q_0^{\rho}(-t),q_0^{\nabla}(-t))\, ,
$$
equation \eqref{eq-hamiltonian-dual} shows that
$$
\overline{\mathcal{H}}^{\nabla}(-t,q_0^{\rho}(-t),q_0^{\nabla}(-t))
=- \partial_{1}L(-t,q_0^{\rho}(-t),q_0^{\nabla}(-t)) \, .
$$
Making $-t=s \in [a,b]_\kappa$ it follows that
$$
\overline{\mathcal{H}}^{\nabla}(s,q_0^{\rho}(s),q_0^{\nabla}(s))=- \partial_{1}L(s,q_0^{\rho}(s),q_0^{\nabla}(s))\, ,
$$
which proves the intended result.
\end{proof}


\section{An Example}
\label{sec:ex}

Let $\mathbb{T} = \left\{0,\frac{1}{8},\frac{1}{4},\frac{3}{8},
\frac{1}{2},\frac{5}{8},\frac{3}{4},\frac{7}{8},1\right\}$
and consider the following problem on $\mathbb{T}$:
\begin{equation}
\label{eq:prb:ex}
\begin{gathered}
\mathcal{I}[q] = \int_0^1 \left[(q^\nabla(t))^2 - 1\right]^2 \nabla t \longrightarrow \min \, ,\\
q(0) = 0 \, , \quad q(1) = 0 \, , \\
q \in C_{ld}^1(\mathbb{T}; \mathbb{R}) \, .
\end{gathered}
\end{equation}
The Euler-Lagrange equation \eqref{eq:nabla:EL:eq} takes the form
\begin{equation}
\label{eq:EL:nwar}
q^\nabla(t) \left[(q^\nabla(t))^2 - 1\right] = \text{const} \, ,
\quad t\in \mathbb{T}_\kappa\, ,
\end{equation}
while our DuBois-Reymond condition for nabla variational problems
(\textrm{cf.} Theorem~\ref{secondEL nabla}) asserts that
\begin{equation}
\label{eq:2ndEL:nwar}
\left[(q^\nabla(t))^2 - 1\right] \left[1 + 3 (q^\nabla(t))^2\right] = \text{const}
\, , \quad t\in \mathbb{T}_\kappa\, .
\end{equation}
The same conservation law \eqref{eq:2ndEL:nwar} is also obtained
from our Noether's theorem for nabla variational problems
(\textrm{cf.} Theorem~\ref{Noether nabla})
since problem \eqref{eq:prb:ex} is invariant
under the family of transformations
$\bar{t} =  t + \epsilon$ and
$\bar{q} = q$, for which $\tau(t,q) \equiv 1$ and $\xi(t,q) \equiv 0$.
Let $\tilde{q}(t) = 0$ for all $t \in \mathbb{T} \setminus \left\{\frac{1}{8}, \frac{7}{8}\right\}$,
and $\tilde{q}\left(\frac{1}{8}\right) = \tilde{q}\left(\frac{7}{8}\right) = \frac{1}{8}$.
One has $\tilde{q}^\nabla\left(\frac{1}{8}\right) = \tilde{q}^\nabla\left(\frac{7}{8}\right) = 1$,
$\tilde{q}^\nabla\left(\frac{1}{4}\right) = \tilde{q}^\nabla(1) = -1$,
and $\tilde{q}^\nabla\left(\frac{i}{8}\right) = 0$, $i = 3,4,5,6$.
We see that $\tilde{q}$ is an extremal, \textrm{i.e.}, it satisfies the
Euler-Lagrange equation \eqref{eq:EL:nwar}. However $\tilde{q}$ cannot
be a solution to the problem \eqref{eq:prb:ex} since it does not satisfy
the DuBois-Reymond condition \eqref{eq:2ndEL:nwar}.
In fact, any function $q$ satisfying $q^\nabla(t) \in \{-1,0,1\}$, $t \in \mathbb{T}_\kappa$,
is an Euler-Lagrange extremal. Among them, only $q^\nabla (t) = 0$ for all $t \in \mathbb{T}_\kappa$
and those with $q^\nabla(t) = \pm 1$ satisfy our condition \eqref{eq:2ndEL:nwar}.
This example shows a problem for which the Euler-Lagrange equation gives
several candidates which are not the solution to the problem, while
the results of the paper give a smaller set of candidates.
Moreover, the candidates obtained from our conservation law lead us directly
to the explicit solution of the problem. Indeed, the null function and any function $q$ with $q(0) = q(1) = 0$ and
$q^\nabla(t) = \pm 1$, $t \in \mathbb{T}_\kappa$, gives $\mathcal{I}[q] = 0$.
They are minimizers because $\mathcal{I}[q] \ge 0$ for any function $q \in C_{ld}^1$.


\section{Final Comments}
\label{sec:fc}

The question of originality of nabla results on time scales after previous delta counterparts
have been proved is an important issue. During recent years several mathematicians
have tried to obtain a satisfactory answer to the problem.
To the best of the authors knowledge there are now six techniques to obtain directly results
for the nabla or delta calculus.
These six approaches were introduced, respectively, in the following references
(ordered by date, from the oldest approach to the most recent one):
\cite{alpha:appr} (the alpha approach); \cite{diam:alpha:appr}
(the diamond-alpha approach); \cite{Aldwoah:appr} (Aldwoah's or
generalized time scales approach); \cite{1:Fasciculi_Mathematici}
(the delta-nabla approach); \cite{Caputo} (Caputo's or duality approach);
\cite{3:2010_CCDC} (the directional approach).
Paper \cite{alpha:appr} introduces the so-called
alpha derivatives, where the $\sigma$ operator in the definition
of delta derivative (or $\rho$ in the definition
of nabla derivative) is substituted by a more general
function $\alpha(\cdot)$; paper \cite{diam:alpha:appr} proposes
a convex combination between delta and nabla derivatives:
$f^{\diamond_{\alpha}}(t)= \alpha f^{\Delta}(t)+(1-\alpha)f^{\nabla}(t)$,
$\alpha \in [0,1]$. Both alpha and diamond-alpha approaches
have been further investigated in the literature,
but they are not effective in the calculus of variations
due to absence of an anti-derivative
(see, \textrm{e.g.}, \cite{diamondBasia,jia}).
Aldwoah's PhD thesis \cite{Aldwoah:appr}
proposes an interesting generalization of the definition of time scale
and develops on it a generalized calculus that gives, simultaneously,
the delta and nabla calculi as particular cases. It provides a very elegant
and general calculus, but it is more complex than all the other approaches.
In some sense Caputo's approach \cite{Caputo} is just a particular case
of Aldwoah's one. More than that, Aldwoah gives all the necessary
formalism and all the proofs, while the duality of Caputo is based on
a principle (the duality principle): ``For any statement true in the nabla
(resp. delta) calculus in the time scale $\mathbb{T}$ there is an equivalent dual statement
in the delta (resp. nabla) calculus for the dual time scale $\mathbb{T}^\ast$.''
Such principle is illustrated in \cite{Caputo} by means of some examples,
but is never proved (it is a principle, not a theorem).
In \cite{3:2010_CCDC} it is studied the
problem of minimizing or maximizing the composition of delta
and nabla integrals with Lagrangians that involve directional derivatives.
In our paper we promote Caputo's technique,
showing how her approach is simple and effective.
To the best of our knowledge, while all the other approaches
(alpha, diamond-alpha, Aldwoah's, delta-nabla, and directional approaches)
have already been further explored in the literature,
Caputo's idea is, to the present moment,
voted to ostracism. Our work contributes to change the state of affairs,
illustrating the duality approach
in obtaining a nabla Noether-type symmetry theorem and
a nabla DuBois-Reymond necessary optimality condition.
The duality approach \cite{Caputo} we are promoting in our work
is not the only approach in the literature; it is not the oldest
or the most recent one; it is also not the most general;
and probably others will appear.
However, the duality approach is simple and beautiful,
making possible short and elegant proofs.

\bigskip

We are grateful to two anonymous referees for
several comments and encouragement words.


\section*{Acknowledgements}

The authors were supported by the Centre for Research on Optimization and Control (CEOC)
from the Portuguese Foundation for Science and Technology (FCT),
cofinanced by the European Community fund FEDER/POCI 2010.



\end{document}